\documentclass[11pt]{amsart}
\usepackage{hyperref}
\usepackage{amssymb, amsmath, amsthm, amsfonts}
\usepackage{verbatim}
\usepackage{bbm}
\usepackage{enumerate}
\usepackage{dsfont}
\usepackage[mathscr]{eucal}
\usepackage{tikz}
\usepackage{color}
\definecolor{gray}{gray}{0.5}

\def\R{\mathbb{R}}
\def\N{\mathbb{N}}
\def\C{\mathbb{C}}
\def\Z{\mathbb{Z}}

\def\beq{\begin{equation}}
\def\endeq{\end{equation}}

\def\mc{\mathcal}

\def\bs{\begin{split}}
\def\es{\end{split}}

\def\mc{\mathcal}

\def\bbone{{\mathbbm 1}}

\newcommand{\ci}[1]{_{{}_{\!\scriptstyle{#1}}}}

\newcommand{\Be}{\begin{equation}}
\newcommand{\Ee}{\end{equation}}

\newcommand{\Bm}{\begin{multline}}
\newcommand{\Em}{\end{multline}}
\newcommand{\Bea}{\begin{eqnarray}}
\newcommand{\Eea}{\end{eqnarray}}
\newcommand{\Beas}{\begin{eqnarray*}}
\newcommand{\Eeas}{\end{eqnarray*}}
\newcommand{\Benu}{\begin{enumerate}}
\newcommand{\Eenu}{\end{enumerate}}
\newcommand{\Bi}{\begin{itemize}}
\newcommand{\Ei}{\end{itemize}}

\def\intslash{\rlap{\kern  .32em $\mspace {.5mu}\backslash$ }\int}
\def\qsl{{\rlap{\kern  .32em $\mspace {.5mu}\backslash$ }\int_{Q_x}}}

\def\N{\mathbb N}

\def\emph#1{{\it #1 }}

\def\ga{\gamma}

\def\inn#1#2{\langle#1,#2\rangle}

\def\meas{{\text{\rm meas}}}

\def\lc{\lesssim}
\def\gc{\gtrsim}

\def\eps{\varepsilon}

\def\fM{{\mathfrak {M}}}
\def\fN{{\mathfrak {N}}}

\def\fS{{\mathfrak {S}}}

\def\bbN{{\mathbb {N}}}

\def\bbR{{\mathbb {R}}}

\def\bbZ{{\mathbb {Z}}}

\def\cC{{\mathcal {C}}}

\def\cH{{\mathcal {H}}}

\def\cK{{\mathcal {K}}}

\def\cM{{\mathcal {M}}}
\def\cN{{\mathcal {N}}}

\def\be#1{\begin{equation}\label{ #1}}
\def\endeq{\end{equation}}
\def\endal{\end{align}}
\def\bas{\begin{align*}}
\def\eas{\end{align*}}
\def\bi{\begin{itemize}}
\def\ei{\end{itemize}}

\def\eps{\varepsilon}

\def\emph#1{{\it #1}}
\def\textbf#1{{\bf #1}}

\def\beq{\begin{equation}}
\def\endeq{\end{equation}}

\def\bs{\begin{split}}
\def\es{\end{split}}

\numberwithin{equation}{section}

\theoremstyle{plain}
\newtheorem{thm}{Theorem}[section]
\newtheorem{prop}[thm]{Proposition}

\newtheorem{lem}[thm]{Lemma}

\newtheorem*{thm*}{Theorem}
\newtheorem*{conj*}{Conjecture}
\newtheorem*{openproblem*}{Open Problem}

\begin{document}
\title[Maximal functions for families of homogeneous curves]{Maximal functions associated with families of homogeneous curves: $\mathbf{L^p}$ bounds for  $\mathbf{p\boldsymbol{\le} 2}$}

\author[S. Guo \  \  \  \  \ \ J. Roos \ \ \ \ \ \ A. Seeger \ \  \ \  \ \ P.-L. Yung] {Shaoming Guo \ \ Joris Roos \ \ Andreas Seeger \ \ Po-Lam Yung}
\address{Shaoming Guo: Department of Mathematics, University of Wisconsin-Madison, 480 Lincoln Dr, Madison, WI-53706, USA}
\email{shaomingguo@math.wisc.edu}

\address{Joris Roos: Department of Mathematics, University of Wisconsin-Madison, 480 Lincoln Dr, Madison, WI-53706,USA}
\email{jroos@math.wisc.edu}

\address{Andreas Seeger: Department of Mathematics, University of Wisconsin-Madison, 480 Lincoln Dr, Madison, WI-53706, USA}
\email{seeger@math.wisc.edu}

\address{Po-Lam Yung: Department of Mathematics, The Chinese University of Hong Kong, Ma Liu Shui, Shatin, Hong Kong \\{\em and } Mathematical Sciences Institute, Australian National University, Canberra ACT 2601, Australia}
\email{plyung@math.cuhk.edu.hk \text{ and } polam.yung@anu.edu.au}

\thanks{S.G. supported in part by NSF grant 1800274.  A.S.  supported  in part by NSF grant 1764295. P.Y. was partially supported by a General Research Fund CUHK14303817 from the Hong Kong Research Grant Council, and a direct grant for research from the Chinese University of Hong Kong (4053341).}


\begin{abstract}
Let  $M^{(u)}$, $H^{(u)}$ be the maximal operator and Hilbert transform along the parabola $(t, ut^2) $. For  $U\subset (0,\infty)$  we  consider $L^p$ estimates for the maximal functions
	$\sup_{u\in U}|M^{(u)} f|$ and
		$\sup_{u\in U}|H^{(u)} f|$, when $1<p\le 2$. The parabolae 
	can be replaced by more general non-flat homogeneous curves.
\end{abstract}

%\date{\today}

\maketitle

\section{Introduction and statement of results}

Let  $b>1$, $u>0$, and $\gamma_b:\R\to\R$ homogeneous of degree $b$, i.e. $\gamma_b(st)=s^b \gamma_b(t)$ for $s>0$. Also suppose $\gamma_b(\pm 1)\neq 0$.
For a Schwartz function $f$ on $\R^2$ we let
\begin{align}\notag
{M}^{(u)}\! f(x)&= \sup_{R>0} \frac {1}{R} \int_{0}^{R} 
|f(x-(t,u\gamma_b(t)))| \, dt,
\\
\notag
{H}^{(u)}\! f(x)&= p.v. \int_{\R} 
f(x-(t,u\gamma_b(t)))\frac{dt}{t},
\end{align}
denote the  maximal function and Hilbert transform of $f$ along  the curve $(t,u\gamma_b(t))$.
For  an arbitrary  nonempty $U\subset (0,\infty)$ we consider the  maximal functions
\Be
\mc{M}^U \!f(x)=\sup_{u\in U} {M}^{(u)}\! f(x),\quad
\mc{H}^U \!f(x)=\sup_{u\in U} |{H}^{(u)}\! f(x)|.
\Ee 

For $2<p<\infty$ the operators $\cM^U$ are bounded on $L^p(\bbR^2)$ for all $U$; this was shown by Marletta and Ricci \cite{MR}.
For the operators $\cH^U$  a corresponding satisfactory  theorem was proved in a previous paper \cite{grsy-first} of the authors. To describe the result let
\[\fN(U)=1+\#\{n\in \bbZ: [2^n, 2^{n+1}] \cap U\neq \emptyset\}.\]
Then, for $2<p<\infty$, $\mc{H}^U $ is bounded on $L^p(\bbR^2)$ if and only if $\fN(U)$ is finite, and we have the equivalence $$c_p \le  \frac{\|\mc{H}^U\|_{L^p\to L^p}} {(\log \fN(U))^{1/2}} \le C_p, \quad 2<p<\infty,$$ with nonzero constants $c_p$, $C_p$. Moreover, for all $p>1$ we have the lower bound $\|\cH^U\|_{L^p\to L^p} \gc \sqrt {\log \fN(U)}$. The consideration of such   results in \cite{grsy-first} and in this paper has multiple motivations. 
First, there is an analogy (although not a close relation) with similar results on  maximal operators  and Hilbert transforms for   families of straight lines; here we mention  the  lower bounds by  Karagulyan \cite{Kar07}, and the currently best upper bounds for $p>2$ by  Demeter and Di Plinio \cite{DD14}. The second motivation comes from the above mentioned work by Marletta and Ricci \cite{MR} on the maximal function for $p>2$, and the third motivation comes from a {\it curved version} of the Stein-Zygmund vector-field problem concerning the $L^p$ boundedness of $M^{(u(\cdot))}$ and $H^{(u(\cdot))}$ where $x\mapsto u(x)$ is a Lipschitz function. 
In this case the  $L^p$ boundedness of $M^{(u(\cdot))}$ for the full range $1<p<\infty$ was proved  by Guo, Hickman, Lie and Roos \cite{GHLR17}, and the analogous result for $H^{(u(\cdot))}$ by  Di Plinio, Guo, Thiele and Zorin-Kranich \cite{DGTZ18}. %More precisely it was proved in \cite{GHLR17} for Lipschitz functions $x\mapsto u(x)$ that $M^{(u(x))}$  defines a bounded operator on all $L^p(\bbR^2)$ for $1<p<\infty$; and \cite{DGTZ18} contains the analogous theorem for the variable Hilbert transform $H^{(u(x)}$. 
We refer to 
the bibliography of \cite{grsy-first} for a list of  related  works.

%and by the analogous questions for Hilbert transform along straight lines (see  \cite{Kar07} for lower bounds,  \cite{DD14} for upper bounds, and 

Regarding the operators $\cM^U$, $\cH^U$ most satisfactory  results (except for certain lacunary sequences) were so far obtained in the range  $p>2$.  In this paper we seek to find efficient upper bounds for the $L^p$ operator norms of $\cM^U$ and $\cH^U$ in the case $1<p\le 2$. It turns out that there is a striking dichotomy between the cases $2<p<\infty$ and $1<p\le 2$. In the latter case, the operator norms of $\cM^U$ and $\cH^U$ depend on an additional quantity that involves the local behavior of the set $U$ on each dyadic interval. 
%We determine a critical exponent $p_{\mathrm{cr}}$ depending on $U$ so that $L^p$ boundedness holds for $p>p_{\mathrm {cr}}$ and fails for $p<p_{\mathrm {cr}}. 
The formulation of the results, using some variant of Minkowski dimension, 
is in part motivated by 
 considerations for spherical maximal functions in the work of Seeger, Wainger, and Wright \cite{SWW95} (see also \cite{SWW97}, \cite{STW03}).

As pointed out in \cite{grsy-first}, with reference to 
\cite{STW03}, $L^p$ boundedness for $p\le 2$ fails, for both $\cM^U$ 
and $\cH^U$,  when $U=[1,2]$; therefore  some additional sparseness condition needs to be imposed.
To formulate such results let, for each  $r>0$ 
$$U^r= r^{-1}U\cap[1,2]= \{\rho\in [1,2]: r\rho\in U\}.$$
For  $0<\delta<1$ we let $N(U^r,\delta)$ 
the $\delta$--covering number of $U^r$, i.e. the minimal number of intervals of length $\delta$ needed to cover $U^r$. 
It is obvious that $\sup_{r>0} N(U^r,\delta)\lc \delta^{-1}$. 
Define
\Be\label{Kp-def}
\cK_p(U,\delta)
 = \delta^{1-\frac 1p} \sup_{r>0} N(U^r, \delta)^{\frac 1p}.
\Ee

Define
\Be
\label{pcr} p_{\mathrm{cr}}(U) = 1+ \limsup_{\delta\to 0+} \frac{ \sup_{r>0} \log N(U^r,\delta)}{\log( \delta^{-1})}
\Ee
Notice  that always $1\le p_{\mathrm{cr}}(U)\le 2$. If  $p_{\mathrm{cr}}(U)<p<2$ there exists an $\eps=\eps(p,U)>0$ such that 
$\sup_{0<\delta<1}  \delta^{-\eps} \cK_p(U,\delta)<\infty$. If $1<p<p_{\mathrm{cr}}(U)$ then there is $ \eps'=\eps'(p,U)>0 $ and a sequence $\delta_n\to 0$ such that
$\limsup_n \delta_n^{ \eps'} \cK_p(U,\delta_n)>0$.

\begin{thm}\label{pmaxthm}

Let 
$1<p\le 2$.

(i) If $p_{\mathrm{cr}}(U)<p\le 2$ then $\cM^U$ is bounded on $L^p(\bbR^2)$.

(ii) If $1< p<p_{\mathrm{cr}}(U)$ then $\cM^U$ is not bounded on $L^p(\bbR^2)$.

(iii)  For every $\eps>0$ we have 
$$c_{p}\sup_{\delta>0} \cK_p(U,\delta)
\le \|\cM^U\|_{L^p\to L^p}\le C_{\eps,p} \sup_{\delta>0} \delta^{-\eps}\cK_p(U,\delta)\,.
$$
Here $c_p, C_{p,\varepsilon}$ are constants only depending on $p$ or $p,\varepsilon$, respectively.
\end{thm}

\begin{thm}\label{psensitivethmH}

Let 
$1<p\le 2$ and $p_{\mathrm{cr}}(U)$ as in \eqref{pcr}.

(i) If $p_{\mathrm{cr}}(U)<p\le 2$ then $\cH^U$ is bounded on $L^p(\bbR^2)$ if and only if $\fN(U)<\infty$.

(ii) If $1< p<p_{\mathrm{cr}}(U)$ then $\cH^U$ is not bounded on $L^p(\bbR^2)$.

(iii)  For every $\eps>0$ we have 
$$
\|\cH^U\|_{L^p\to L^p}\le C_p 
\sqrt{\log (\fN(U))} +C_{\eps,p}\sup_{\delta>0} \delta^{-\eps}\cK_p(U,\delta).
$$ and 
$$
c_p \big(
\sqrt{\log (\fN(U))} +\sup_{\delta>0} \cK_p(U,\delta)\big)\le
\|\cH^U\|_{L^p\to L^p}.
$$
Here $c_p, C_p, C_{p,\varepsilon}$ are constants only depending on $p$ or $p,\varepsilon$, respectively.
\end{thm}
We note that part (i), (ii) of the theorems follow immediately from part (iii) of the respective theorem. 

We discuss some examples. We have  $p_{\mathrm{cr}}(U)=1$ for lacunary $U$ and  we have  $p_{\mathrm{cr}}(U)=2$ if $U$ contains any intervals. There are many interesting intermediate examples with 
$1<p_{\mathrm{cr}}(U)<2$, 
see  \cite{SWW95}. One may take for $U$ a self similar Cantor set $\cC_\beta$ of Minkowski dimension $\beta$, contained in $[1,2]$; then $p_{\mathrm{cr}}(\cC_\beta)= 1+ \beta$. 
This remains true if for 
$U$ we take $\cup_{k\in \bbZ}  2^k \cC_\beta$ in Theorem \ref{pmaxthm}, or, with finite $F\subset \bbZ$,
we take 
$U=\cup_{k\in F}  2^k \cC_\beta$ in Theorem \ref{psensitivethmH}.

Another set of examples comes from considering convex sequences. 
One may take $S_a
=\{1+n^{-a}: n\in \bbN\}$ then $p_{\mathrm cr}(S_a)=\frac{2+a}{1+a}$. Again we may also take suitable unions of dilates of $S_a$, i.e. for $U$ 
we can take $\cup_{k\in \bbZ}  2^k S_a$ in Theorem \ref{pmaxthm}, or, 
$U=\cup_{k\in F}  2^k S_a$ in Theorem \ref{psensitivethmH},  provided that $F\subset \bbZ$ is  finite.

We shall in fact prove  sharper but more technical versions of Theorems \ref{pmaxthm} and  \ref{psensitivethmH}.
The term $C_{\eps, p}  \delta^{-\eps}\cK_p(U,\delta)$ can be replaced with one with  logarithmic dependence, namely 
$$C_p [\log(2/\delta)]^{A}\cK_p(U,\delta)$$  for  $A>14/p-6$.
More precisely,  we have the following 

\begin{thm}\label{logsthm}
Let 
$1<p\le 2$. Then there is $C$ independent of $p$ and $U$ so that
\Be\label{MUlogest}
\|\cM^U\|_{L^p\to L^p} \le C\sum_{\ell\ge 1} \vartheta_{p,\ell} \cK_p(U, 2^{-\ell}),
\Ee
where $\vartheta_{p,\ell}=
(p-1)^{3-\frac {10}p}$ if $\ell\le (p-1)^{-1}$ and 
$\vartheta_{p,\ell}=
 \ell^{7(\frac {2}p -1)}$ if $\ell> (p-1)^{-1}$. 
 
 Moreover,
%(p-1)^{3-\frac {10}p}\bbone_{\ell\le (p-1)^{-1}} + \ell^{7(\frac {2}p -1)}\bbone_{\ell> (p-1)^{-1}}$
\Be\label{HUlogest}
\|\cH^U\|_{L^p\to L^p} \le
C (p-1)^{-7} \sqrt{\log (\fN(U))} +
C (p-1)^{-2} \sum_{\ell\ge 1} \vartheta_{p,\ell} \cK_p(U, 2^{-\ell}).
\Ee
\end{thm}

%Interesting intermediate examples for sets $U$ with include Cantor sets and convex sequences; 

{\emph{Structure of the paper.}} In \S \ref{sec:decomp} we decompose the operators $\mathcal{M}^U$, $\mathcal{H}^U$ in the spirit of \cite{grsy-first} in order to prepare for the proof of Theorem \ref{logsthm}. The proof of Theorem \ref{logsthm} is then completed in \S \ref{sec:bootstrap} and \S \ref{sec:mainthmpf}. Finally, the lower bounds claimed in Theorem \ref{pmaxthm} and Theorem \ref{psensitivethmH} are addressed in \S \ref{Mink-lower-bounds}.\\

\section{Basic reductions}\label{sec:decomp}
We recall some notation and basic reductions from \cite{grsy-first}.
By the assumption of homogeneity and $\gamma_b(\pm 1)\neq 0$ there are $c_\pm\neq 0$ such that $\gamma_b(t)=
c_+ t^b$ for   $t>0$, and
 $\gamma_b(t)=
c_- (-t)^b$ for   $t<0$, and finally $\gamma_b(0)=0$.
We note that by scaling we may always assume that $c_-=1$. Let $\chi_+\in C^\infty_c$ be supported in $(1/2,2)$ such that $$\sum_{j\in \bbZ} \chi_+(2^j t) =1 \text{ for $t>0$.} $$
Let $\chi\ci-(t)=\chi_+(-t)$ and $\chi=\chi\ci+ + \chi\ci-$.
We define measures $\tau_0$, $\sigma_0$, $\sigma_\pm$ by
\begin{align*}\label{sigmapmdef}
\inn{\tau_0} {f} &= \int f(t,\ga_b(t)) \chi\ci{+}(t) dt,
\\
\inn{\sigma\ci\pm} {f} &= \int f(t,\ga_b(t)) \chi\ci{\pm}(t) \frac {dt} {t},
\\
\sigma_0 &= \sigma_+ + \sigma_-.
\end{align*}
Let, for $j\in \bbZ$,   the measures $\tau_j^u, \sigma_j^u$ be defined by 
\begin{align*}\inn{\tau_j^u  } {f} &= \int f(t,u\ga_b(t)) 2^j \chi_+(2^j t) dt,\\
\inn{\sigma_j^u  } {f} &= \int f(t,u\ga_b(t)) \chi(2^j t) \frac {dt} {t}.
\end{align*}
By homogeneity of $\gamma_b$ we have
$\tau_j^u = 2^{j(1+b)} \tau_0^u (\delta_{2^j}^b \cdot)$ with  $\delta_{t}^bx= (tx_1, t^bx_2)$, as well as the analogous relation between $\sigma_j^u $ and $\sigma_0^u$. We note that the $\tau_j^u$ are positive measures and the $\sigma_j^u$ have cancellation.

For Schwartz functions $f$ the Hilbert transform along $\Gamma_b^u$ can be written as
$$H ^{(u)}f= \sum_{j\in \bbZ}{\sigma_j^u}* f.$$
For the maximal function it is easy to see that there is the pointwise estimate
\Be\label{ptwmaxest}
M^{(u)}f(x)\le C \sup_{j\in \bbZ} \tau_j^u*|f|.
\Ee

Following \cite[\S 2]{grsy-first} we  further decompose $\sigma_0$ and $\tau_0$.
Choose Schwartz function $\eta_0$,  supported in $\{|\xi|\le 100\}$  and equal with $\eta_0(\xi)=1$ for $|\xi|\le 50$.
Let $\varsigma_+\in C^\infty_c(\bbR)$ be supported in 
$(b (1/4)^{b-1} , b 4^{b-1} )$ and  equal to $1$ on 
$[b (2/7)^{b-1} , b (7/2)^{b-1} ]$.
Let $\varsigma_-\in C^\infty_c(\bbR)$ be supported on
$(-b 4^{b-1} , -b (1/4)^{b-1} )$ and equal to $1$ on 
$[-b (7/2)^{b-1} , -b (2/7)^{b-1} ]$.

One then decomposes 
	\begin{align*} 
	\sigma_0&= \phi_0+\mu_{0,+} +\mu_{0,-}
	\\
	\tau_0&= \varphi_0+\rho_{0}
	\end{align*}
	where $\phi_0$, $\varphi_0$ are  given by
	\Be\begin{aligned}\nonumber
		\widehat{ \phi_0}(\xi) =
		\eta_0(\xi) \widehat \sigma_0(\xi) 
		&+ (1-\eta_0(\xi)) \big(1- \varsigma_- (\tfrac{\xi_1}{c_+\xi_2})\big)\widehat \sigma_+(\xi) 
		\\
		&+ (1-\eta_0(\xi)) \big(1- \varsigma_+(\tfrac{\xi_1}{c_-\xi_2})\big)\widehat \sigma_-(\xi) 
	\end{aligned}\Ee
	and 
	\Be \nonumber
		\widehat{ \varphi_0}(\xi) =
		\eta_0(\xi) \widehat \tau_0(\xi) 
		+ (1-\eta_0(\xi)) \big(1- \varsigma_- (\tfrac{\xi_1}{c_+\xi_2})\big)\widehat \tau(\xi) .
		\Ee

	The measures 
	and $\mu\ci{0,\pm}$ and $\rho_0$ are given via the Fourier transform by
	\begin{align*}
	\widehat \mu_{0,+} (\xi) &= (1-\eta_0(\xi))  \varsigma_- (\tfrac{\xi_1}{c_+\xi_2})\widehat \sigma_+(\xi) ,
	\\
	\widehat \mu_{0,-} (\xi) &= (1-\eta_0(\xi))  \varsigma_+(\tfrac{\xi_1}{c_-\xi_2})\widehat \sigma_-(\xi) 
	\end{align*}
	and
	\begin{equation}\label{rho0def}
	\widehat \rho_{0} (\xi) = (1-\eta_0(\xi))  \varsigma_- (\tfrac{\xi_1}{c_+\xi_2})\widehat \tau_0(\xi). 
	\end{equation}

As in Lemma 2.1 of \cite{grsy-first},
the functions $\varphi_0$, $\phi_0$ are Schwartz functions. In addition we have 
	$\widehat \phi_0(0)=0$. 

Define, for $j\in \bbZ$,  $\varphi_j$ and  $\phi_j$ 
by scaling via $\widehat \varphi_j(\xi)=\widehat \varphi_0 (2^{-j}\xi_1, 2^{-jb}\xi_2)\widehat f(\xi) $ and
$\widehat \phi_j(\xi)=\widehat \phi_0 (2^{-j}\xi_1, 2^{-jb}\xi_2)\widehat f(\xi)$.
Define $A_{j,0}^u f$  by
$$\widehat{A_{j,0}^u f} (\xi) = \widehat \varphi_j(\xi_1, u\xi_2) \widehat f(\xi)$$
and let $\cM_0 f(x) =\sup_{j\in \bbZ} \sup_{u\in \bbR} |A_{j,0}^u f(x) | $. Let 
$$\widehat{S^{(u)} f} (\xi) =\sum_{j\in \bbZ} 
\widehat \phi_j (\xi_1, u \xi_2)\widehat f(\xi).
$$
Let $M^{\mathrm{str}} f$ denote the strong maximal function of $f$. 
For $p\in (1,2]$ we have 
\Be\label{Mstrong} \|M^{\mathrm{str}}\|_{L^p\to L^p} \le C(p-1)^{-2}. \Ee
This follows from the pointwise bound $M^{\mathrm{str}}\le M^{(1)}\circ M^{(2)}$, where $M^{(k)}$ denotes the Hardy--Littlewood maximal operator taken in the $k$th variable. Indeed, $M^{(k)}$ is of weak type $(1,1)$ so Marcinkiewicz interpolation gives $\|M^{(k)}\|_{L^p\to L^p} \le C (p-1)^{-1}$ for some constant $C>0$ and all $p\in (1,2]$, which implies \eqref{Mstrong}.

\begin{lem} \label{ell=zero-lem}
There exists a constant $C$ such that for all $p \in (1,2]$,\\
(i) $$\|\cM_0 f\|_p \le C (p-1)^{-2} \|f\|_p.$$
(ii) $$\|\sup_{u\in U} |S^{(u)} f| \|_p \le C (p-1)^{-7}\sqrt{\log\fN(U)} \|f\|_p.
$$
\end{lem}

\begin {proof}
Part (i) follows from the estimate 
\Be \label{ell=0Mstr}
|A_{j,0}^u f(x)| \le C M^{\mathrm{str}} f(x).
\Ee 
Part (ii) is more substantial and relies on the Chang--Wilson--Wolff bounds for martingales, \cite{CWW85}. This is the subject of Theorem 2.2 in \cite{grsy-first}. The dependence on $p$ was not specified there, but can be obtained by a literal reading of the proof provided in \cite[\S 4]{grsy-first}. We remark that the exponent $7$ can likely be improved, but it is satisfactory for our purposes here.
\end{proof}

We also  decompose $\widehat {\rho_0}$ and $\widehat {\mu_{0,\pm}}$ further by making an isotropic 
decomposition for large frequencies. Let $\zeta_0\in C^\infty_c(\bbR^2)$ supported in 
$\{\xi:|\xi|<2\}$ and such that $\zeta_0(\xi)=1$ for $|\xi|\le 5/4$. For $\ell=1,2,3,\dots$ let 
\[\zeta_\ell(\xi)=\zeta_0(2^{-\ell}\xi)- \zeta_0(2^{1-\ell}\xi).\] Then for $\ell>0$, $\zeta_\ell$ is supported in the annulus $\{\xi:2^{\ell-1}<|\xi|<2^{\ell+1} \}$ and we have 
$1=\sum_{\ell>0}\zeta_\ell(\xi)$ for $\xi$ in the support of $\widehat{\rho_0}, \widehat{\mu_{0,\pm}}$.

Define operators $A_{j,\ell}^u $ and $T_{j,\ell,\pm}^u$ by
\begin{align}\label{Ajludef}
\widehat {A_{j,\ell}^u f}(\xi)&= \zeta_\ell (2^{-j}\xi_1, 2^{-jb} u\xi_2) \widehat{\rho_{0}}(2^{-j}\xi_1, 2^{-jb} u\xi_2)  \widehat{f}(\xi),\\\label{Tjludef}
\widehat {T_{j,\ell,\pm}^uf}(\xi)&= \zeta_\ell (2^{-j}\xi_1, 2^{-jb} u\xi_2)\widehat{\mu_{0,\pm}}(2^{-j}\xi_1, 2^{-jb} u\xi_2)  \widehat{f}(\xi).
\end{align}

We shall show 
\begin{prop} \label{mainthm}
There is $C>0$ such that for 
 each $\ell>0$, $p \in (1,2]$ we have
 \begin{equation}\label{Ajlmax}
\big\| \sup_{u\in U} \sup_{j\in \bbZ}| A_{j,\ell}^{u} f|\,\big\|_p  \le C  \vartheta_{p,\ell} \mathcal{K}_p(U,2^{-\ell}) \|f\|_p,
\end{equation} 
where $\vartheta_{p,\ell}=(p-1)^{3-\frac {10}p}\bbone_{\ell\le (p-1)^{-1}} + \ell^{7(\frac {2}p -1)}\bbone_{\ell> (p-1)^{-1}}$ and
\begin{equation}\label{Tjl}
\Big\| \sup_{u\in U} \Big|\sum_{j\in \bbZ}
T_{j,\ell,\pm}^{u} f\Big|\,\Big\|_p \le C (p-1)^{-2}  \vartheta_{p,\ell} \mathcal{K}_p(U,2^{-\ell}) \|f\|_p.
\end{equation}
\end{prop}

We claim that Proposition \ref{mainthm} implies Theorem \ref{logsthm}. 
Indeed, we have for non--negative $f$,
\[ \mathcal{M}^U f \lesssim \mathcal{M}_0 f + \sum_{\ell>0} \sup_{u\in U} \sup_{j\in\Z} |A_{j,\ell}^{u} f| \]
and thus \eqref{MUlogest} follows from part (i) of Lemma \ref{ell=zero-lem} and \eqref{Ajlmax}.
It remains to show \eqref{HUlogest}. But in view of the decomposition,
\[ H^{(u)} = S^{(u)} + \sum_{\pm}\sum_{\ell>0} \sum_{j\in\Z} T_{j,\ell,\pm}^u,  \]
this follows from part (ii) of Lemma \ref{ell=zero-lem} and \eqref{Tjl}. This finishes the proof of Theorem \ref{logsthm}.

We conclude this section with some estimates that will be used in the proof of Proposition \ref{mainthm}. 
We will harvest the required decay in $\ell$ from the following simple estimate.
For $p\in[1,2]$, $\ell>0$, $j\in\Z$, $u\in (0,\infty)$ we have
\begin{equation}\label{eqn:vdC}
\|A^{u}_{j,\ell} f\|_p \le C 2^{-\ell (1-1/p)} \|f\|_p.
\end{equation}

Indeed, the endpoint $p=2$ is a consequence of Plancherel's theorem and van der Corput's lemma, while $p=1$ follows because the convolution kernel of $A_{j,\ell}^u f$ is $L^1$--normalized.
Another key ingredient will be the following pointwise estimate. From the definition of $A_{j,\ell}^u$ in \eqref{Ajludef} we have for $\ell>0$, $j\in \bbZ$, $u\in (0,\infty)$ that
\begin{equation}\label{eqn:pointwiseA}
|A_{j,\ell}^u f | \le C M^{\mathrm{str}} (\tau_j^u*|f|).
\end{equation}
This follows because we have
\[ A_{j,\ell}^u f = (f * \tau_j^u) * \kappa_{j,\ell}^u,  \]
with $\kappa_{j,\ell}^u$ certain Schwartz functions that can be read off from the definitions \eqref{rho0def}, \eqref{Ajludef} and satisfy $|f*\kappa_{j,\ell}^u|\le C M^\mathrm{str} f$ with $C>0$ not depending on $j,\ell,u$.

We also need to introduce appropriate Littlewood--Paley decompositions.
Let $\chi^{(1)} $ be an even $C^\infty$ function supported on 
\[\{\xi_1: |c_+|b 2^{-3b-1} \le |\xi_1|\le |c_+|b 2^{3b+1}\}\]
and equal to $1$ for 
$ |c_+|b 2^{-3b} \le |\xi_1|\le |c_+|b 2^{3b}$.
Let $\chi^{(2)} $ be an even $C^\infty$ function supported on 
\[\{\xi_2: 2^{-2b-1} \le |\xi_2|\le2^{2b+1}\}\]
and equal to $1$ for 
$ 2^{-2b} \le |\xi_2|\le2^{2b}$.
Define $P^{(1)}_{k_1,\ell} $, $P^{(2)}_{k_2,\ell,b} $ by
\begin{align*}
\widehat{P^{(1)}_{k_1,\ell} f}(\xi)&=  \chi^{(1)} (2^{-k_1-\ell}\xi_1) \widehat f(\xi)
\\
\widehat{P^{(2)}_{k_2,\ell,b} f}(\xi)&=  \chi^{(2)} (2^{-k_2b-\ell}\xi_2) \widehat f(\xi)
\end{align*} 
Then  for $s\in[1,2^b]$, 
\Be\label{LPrepr}
 A_{j,\ell}^{2^{bn} s}=A_{j,\ell}^{2^{bn}s }
P^{(2)}_{j-n,\ell, b}  P^{(1)}_{j,\ell} 
=
 P^{(1)}_{j,\ell} 
 P^{(2)}_{j-n,\ell, b} 
A_{j,\ell}^{2^{bn} s}.
\Ee

For $p\in(1,2]$ we have the Littlewood--Paley inequalities 
\Be \label{standardLP}
\Big\|\Big(\sum_{k_1\in \bbZ}\sum_{k_2\in \bbZ}
\big|
P^{(1)}_{k_1,\ell} P^{(2)}_{k_2,\ell,b} f\big|^2\Big)^{1/2}
\Big\|_p \le C (p-1)^{-2} \|f\|_p
\Ee
and 
\Be \label{dualLP}
\Big\|\sum_{k_1\in \bbZ}\sum_{k_2\in \bbZ}
P^{(1)}_{k_1,\ell} P^{(2)}_{k_2,\ell,b} f_{k_1,k_2} 
\Big\|_p \le C(p-1)^{-2}
\Big\|\Big(\sum_{k_1\in\Z}\sum_{k_2\in\Z} |f_{k_1,k_2}|^2\Big)^{1/2}\Big\|_p,
\Ee
which also hold for Hilbert space valued functions. Similarly as in \eqref{Mstrong}, each of these two inequalities follows from two applications of appropriate one-dimensional Littlewood--Paley inequalities and the fact that these come with a constant of $(p-1)^{-1}$ each, owing to Marcinkiewicz interpolation with the weak $(1,1)$ endpoint.

\section{A positive bilinear operator}\label{sec:bootstrap}

In this section we are given for every $n\in \bbZ$ an at most  countable set \[ \fS(n) =\{s_{n}(i): i=1,2,\dots \} \,\subset \,  [1,2^b]. \] 

\begin{prop} \label{positiveopprop}
There is a constant $C$ independent of the choice of the sets $\fS(n)=\{s_{n}(i)\}$, $n\in \bbN$,  such that for $1<p\le 2$ and $\ell>0$,
\begin{multline*} 
\Big\| \Big(\sum_{j,n\in\Z} \sum_{i\in\N}\big |w_n(i)\,
\mathcal{A}_{j,\ell}^{2^{bn} s_{n}(i) }
f\big|^2\Big)^{1/2} \Big\|_p
\\\le C (p-1)^{3-\frac {10}p} 2^{-\ell (p-1)/2} \sup_{n\in\Z}\|w_n\|_{\ell^p} \|f\|_p
\end{multline*}
for all functions $f$ and $w_n:\N\to\C$. This holds for $\mathcal{A}_{j,\ell}^{2^{bn} s_n(i)}$ being any one of the following:
\[ A^{2^{bn} s_n(i)}_{j,\ell},\; 2^{-\ell}\frac{d}{ds} A^{2^{bn} s}_{j,\ell} |_{s=s_n(i)},\; T^{2^{bn} s_n(i)}_{j,\ell,\pm},\; 2^{-\ell} \frac{d}{ds}T^{2^{bn} s}_{j,\ell,\pm} \big|_{s=s_n(i)}. \]
\end{prop}

We will only detail the proof in the case $\mathcal{A}_{j,\ell}^{2^{bn} s_n(i)}=A^{2^{bn} s_n(i)}_{j,\ell}$. The other cases follow \emph{mutatis mutandis}. To this end note that the corresponding variants of the main ingredients \eqref{eqn:vdC}, \eqref{eqn:pointwiseA}, \eqref{LPrepr} also hold for each of the other cases, the underlying reasoning being identical in each case.

In the proof of the proposition we use a bootstrapping argument by Nagel, Stein and Wainger 
\cite{NSW} in a simplified and improved form given in unpublished work by Christ (see \cite{Carbery88} for an exposition).

We first introduce an auxiliary maximal operator. For $R\in \bbN$ let
$$\fM_R[f,w](x)= \sup_{-R\le j,n\le R} \sup_{i\in \bbN} \big|w_n(i)\,\tau_j^{2^{bn} s_{n}(i) } 
\!*\!f(x)\big|.
$$
We let $B_p(R)$ be the best constant $C$ in the inequality
\[\|\fM_R[f,w]\|_p
\le C
\sup_{n\in\Z}\|w_n\|_{\ell^p}  \|f\|_p,\]
that is,
\Be\label{BpR}
B_{p}(R) =\sup\{ \|\fM_R[f,w] \|_p: \,\|f\|_p\le 1, \,\sup_{n\in\Z} \|w_n\|_{\ell^p} \le 1\}.
\Ee 
The positive number $B_p(R)$ is finite, as from the 
uniform $L^p$-boundedness of the operator
$f\mapsto \tau_j^u*f$  we have  $B_p(R)\le C(2R+1)^{2/p}$. It is our objective to show that $B_p(R)$ is independent of $R$. 
More precisely, we claim that there is a constant $C$ independent of the choice of the sets $\fS(n)$,  such that for $1<p\le 2$,
\Be \label{positiveoppropstatement}
B_p(R) \le C (p-1)^{2-10/p}.
\Ee

We begin with an estimate for a vector--valued  operator.

\begin{lem} \label{positive-prel-lemma}
Let $1<p\le 2$, $p\le q\le\infty$. Then
\begin{multline} \label{positive-prel}
\Big\|\Big(\sum_{-R\le j,n\le R} \sum_{i\in\N} |w_n(i)\, A^{2^{bn} s_n(i)}_{j,\ell} 
g_{j,n} |^q\Big)^{1/q} \Big\|_p
\\ \le C (p-1)^{-2(1-\frac pq)}
  B_p(R) ^{1-\frac pq} 2^{-\ell (1-\frac 1p) \frac pq} 
\sup_{n\in\Z}\|w_n\|_{\ell^p} \Big\|\Big(\sum_{j,n\in\Z} |g_{j,n} |^q\Big)^{1/q}\Big\|_p
\end{multline}
\end{lem}
\begin{proof} 
The case  $q=p$  of \eqref{positive-prel} follows from \eqref{eqn:vdC}. For $q=\infty$ we use \eqref{eqn:pointwiseA} to estimate
\begin{align*} 
&\big\|\sup_{-R\le j,n\le R}\sup_{i\in\N} |w_n(i)\, A^{2^{bn} s_n(i)}_{j,\ell} g_{j,n} |\big\|_p
\\&\le C \big\|\sup_{-R\le j,n \le R}\sup_{i\in\N} |w_n(i)|\, M^{\mathrm{str}} [\tau^{2^{bn} s_n(i)}_{j}\!*\! 
|g_{j,n}| ]\big\|_p
\\&\le C \big\|M^{\mathrm{str}} \big[\sup_{-R\le j,n\le R}\sup_{i\in\N} |w_n(i)|\,  \tau^{2^{bn} s_n(i)}_{j} \!*\!(\sup_{j',n'\in\Z}|g_{j',n'}|) \big]\big\|_p
\end{align*}
where we have used the positivity of the operators $f\mapsto \tau_j^u*f$. By \eqref{Mstrong}
we can dominate the last displayed expression  by
\begin{align*}
 & C'(p-1)^{-2}  \big\|\sup_{-R\le j,n\le R}\sup_{i\in\N} |w_n(i)|\,  \tau^{2^{bn} s_n(i)}_{j}\!*\! [\sup_{j',n'\in\Z}|g_{j',n'}|] \big\|_p
\\
&\lc (p-1)^{-2} B_p(R) \,\sup_{n\in\Z}\|w_n\|_{\ell^p}
\big\|
\sup_{j',n'\in\Z}|g_{j',n'}| \big\|_p
\end{align*}
which establishes the case $q=\infty$. The case $p<q<\infty$ follows by interpolation.
\end{proof}

\begin{proof}[Proof of Proposition \ref{positiveopprop}]
We use the decomposition $\tau_j^u*f= \sum_{\ell=0}^\infty
A^u_{j,\ell} f$.
By \eqref{ell=0Mstr} 
we get
\Be\notag
\Big\| \sup_{j,n\in\Z} \sup_{i\in \bbN}\big |w_n(i)\,
A_{j,0}^{2^{bn} s_{n}(i) } 
f\big| \Big\|_p
\lc  (p-1)^{-2}  \sup_{n\in\Z}
\|w_n\|_{\ell^\infty}  \|f\|_p.
\Ee
For $\ell>0$ we have, 
\[\Big\| \sup_{-R\le j,n\le R} \sup_{i\in \bbN}\big |w_n(i)\,
A_{j,\ell}^{2^{bn} s_{n}(i) }
f\big| \Big\|_p
\le \Big\| \Big(\sum_{-R\le j,n\le R} \sum_{i\in\N}\big |w_n(i)\,
A_{j,\ell}^{2^{bn} s_{n}(i) }
f\big|^2\Big)^{1/2} \Big\|_p\]
and, by  \eqref{LPrepr} and Lemma \ref{positive-prel-lemma} for $q=2$, and \eqref{standardLP},
\begin{align} \label{sqrfctbdAjl}
&\Big\| \Big(\sum_{-R\le j,n\le R} \sum_{i\in\N}\big |w_n(i)\,
A_{j,\ell}^{2^{bn} s_{n}(i) }
f\big|^2\Big)^{1/2} \Big\|_p
\\
\notag
&\lc (p-1)^{-2(1-\frac p2)} B_p(R)^{1-\frac p2}2^{-\ell(1-\frac 1p)\frac p2}
\sup_{n\in\Z}\|w_n\|_{\ell^p} 
\Big\|\Big(\sum_{j,n\in\Z} \big|P^{(2)}_{j-n,\ell, b}  P^{(1)}_{j,\ell} f\big|^2\Big)^{1/2}\Big\|_p
\\
\notag
&\lc (p-1)^{p-4} 2^{-\ell (p-1)/2} B_p(R)^{1-p/2} 
\sup_{n\in\Z}\|w_n\|_{\ell^p} \|f\|_p.
\end{align}
This implies, for $1<p\le 2$ 
\begin{align*} 
B_p(R) &\lc \Big[(p-1)^{-2} +\sum_{\ell>0}
(p-1)^{p-4} 2^{-\ell (p-1)/2} B_p(R)^{1-p/2} \Big]
\\
&\lc (p-1)^{-2} + (p-1)^{p-5} B_p(R)^{1-p/2}
\end{align*} 
which leads to
\[ B_p(R)\lc (p-1)^{2-10/p}.\]
If we use this inequality in \eqref{sqrfctbdAjl}
and observe \[p-4+(2-10/p)(1-p/2)=3-10/p,\] then the claimed inequality in Proposition \ref{positiveopprop} follows by the monotone convergence theorem.
\end{proof}

\section{Proof of Proposition \ref{mainthm}}\label{sec:mainthmpf}
For $n\in \bbZ$ let $U_n\subset [1, 2^b] $
be defined by $$U_n= \{ 2^{-bn} u: u\in [2^{bn}, 2^{b(n+1)}]\cap U\}$$ and let 
$$\cN_{n,\ell}(U)= \#\big\{k:  [2^{-\ell} k, 2^{-\ell}(k+1))\cap U_n\neq \emptyset \big\}.
$$
Then we have
\[ 2^{-\ell (1-\frac1p)} \sup_{n\in\Z} \mathcal{N}_{n,\ell}(U) \approx \cK_p(U, 2^{-\ell}) \]
We cover each set $U_n$ with dyadic intervals of the form $$I_{k,\ell}=[k2^{-\ell}, (k+1)2^{-\ell})$$ where $k\in \bbN$.
Denote by $\fS_{n,\ell}$ 
the left endpoints of these intervals and note $\cN_{n,\ell}(U) = \#\fS_{n,\ell}$. We label the set of points in $\fS_{n,\ell} $, by 
$\{s_{n,\ell}(i)\}_{i=1}^{\cN_{n,\ell}(U)}$ and write
\begin{align*} &\sup_{j\in\Z} \sup_{u\in U} |A_{j,\ell}^u f(x)| = \sup_{j\in\Z} \sup_{n\in\Z} \sup_{s\in U_n} |A_{j,\ell}^{2^{nb} s} f(x)| 
\\&\le \sup_{j,n\in\Z} \sup_{i=1,\dots \cN_{n,\ell}(U)} 
|A^{2^{nb}s_{n,\ell}(i)}_{j,\ell} f(x)| 
\\&\qquad +
\sup_{j,n\in\Z} \sup_{i=1,\dots \cN_{n,\ell}(U)} \int_0^{2^{-\ell} } \Big|\frac{d}{d\alpha} 
A^{2^{nb}(s_{n,\ell}(i)+\alpha)}_{j,\ell} f(x)\Big| d\alpha.\end{align*}
Hence 
\begin{multline*}
\Big\| \sup_{j\in\Z}\sup_{u\in U}
 |A_{j,\ell}^u f| \Big\|_p
\, \le\Big\| \Big(\sum_{j,n\in\Z}\sum_{i=1}^{\cN_{n,\ell}(U)} |A_{j,\ell}^{2^{nb} s_{n,\ell}(i) } f|^2\Big)^{1/2}\Big\|_p
\\+ \int_{0}^{2^{-\ell} }
\Big\|
\Big(\sum_{j,n\in\Z}\sum_{i=1}^{\cN_{n,\ell}(U)} \Big| \frac{d}{d\alpha} A_{j,\ell}^{2^{nb} (s_{n,\ell}(i) +\alpha) } f|^2\Big)^{1/2}\Big\|_p d\alpha
\end{multline*}
and by part (ii) of Proposition \ref{positiveopprop} 
 both expressions on the right hand side can be estimated by
 \Be\label{pnear1} C(p-1)^{3-10/p} 2^{-\ell(p-1)/2}\sup_{n\in\Z} \cN_{n,\ell}(U)^{1/p} \|f\|_p.
 \Ee
 This estimate is efficient for $1<p<1+\ell^{-1}$.
 Note that in this range $2^{-C\ell (1-1/p)}\approx 1$ and $\cN_{n,\ell}(U)^{1/p}\approx \cK_p(U, 2^{-\ell})$. For $p=2$ we have the inequality
 \begin{align}
 &\Big\| \Big(\sum_{j,n\in\Z}\sum_{i=1}^{\cN_{n,\ell}(U)} |A_{j,\ell}^{2^{nb} s_{n,\ell}(i) } f|^2\Big)^{1/2}\Big\|_2
 \label{L2endpt}
\\&\qquad+ \int_{0}^{2^{-\ell} }
\Big\|
\Big(\sum_{j,n\in\Z}\sum_{i=1}^{\cN_{n,\ell}(U)} \Big| \frac{d}{d\alpha} A_{j,\ell}^{2^{nb} (s_{n,\ell}(i) +\alpha) } f|^2\Big)^{1/2}\Big\|_2 d\alpha
\notag
 \\&\lc 2^{-\ell/2 } \sup_{n\in\Z} \cN_{n,\ell}(U)^{1/2} \|f\|_2.
 \notag
\end{align}

For $p_\ell:=1+\ell^{-1} <p<2$ we use the Riesz--Thorin interpolation theorem
(together with the  fact that $ (p_\ell-1) ^{C/\ell}\approx_C 1$ and $(p_\ell-1)^{-A}=\ell^A$).
We then obtain for $p_\ell<p<2$
 \begin{align}
 &\Big\| \Big(\sum_{j,n\in\Z}\sum_{i=1}^{\cN_{n,\ell}(U)} |A_{j,\ell}^{2^{nb} s_{n,\ell}(i) } f|^2\Big)^{1/2}\Big\|_p
 \notag
\\ &\qquad+ \int_{0}^{2^{-\ell} }
\Big\|
\Big(\sum_{j,n\in\Z}\sum_{i=1}^{\cN_{n,\ell}(U)} \Big| \frac{d}{d\alpha} A_{j,\ell}^{2^{nb} (s_{n,\ell}(i) +\alpha) } f|^2\Big)^{1/2}\Big\|_p d\alpha
\notag
 \\&\lc 2^{-\ell(1-\frac 1p) } \sup_{n\in\Z} \cN_{n,\ell}(U)^{1/p} \ell^{7(\frac 2p-1)} \|f\|_p.
 \label{Lpintermed}
\end{align}
Thus we have established \eqref{Ajlmax}.
The proof of \eqref{Tjl} 
is similar but the reduction to a square--function estimate requires one more use of a Littlewood--Paley estimate. We have, using the analogue of \eqref{LPrepr} for $T^{2^{bn} s}_{j,\ell,+}$ 
\begin{align*}
&\Big\| \sup_{n\in\Z}\sup_{u\in U\cap[2^{nb}, 2^{(n+1)b}]}
 \Big|\sum_{j\in\Z} T_{j,\ell,+}^u f\Big| \Big\|_p
\\
&\le\Big\| \Big(\sum_{n\in\Z}\sum_{i=1}^{\cN_{n,\ell}}\Big |\sum_{j\in\Z} P^{(1)}_{j,\ell} P^{(2)}_{j-n, \ell, b}
 T_{j,\ell,+}^{2^{nb} s_{n,\ell}(i) } f\Big|^2\Big)^{1/2}\Big\|_p
\\ &\quad + \int_{0}^{2^{-\ell} }
\Big\|
\Big(\sum_{n\in\Z}\sum_{i=1}^{\cN_{n,\ell}(U)} \Big|\sum_{j\in \bbZ}P^{(1)}_{j,\ell} P^{(2)}_{j-n, \ell, b} \frac{d}{d\alpha} T_{j,\ell,+}^{2^{nb} (s_{n,\ell}(i) +\alpha) } f\Big|^2\Big)^{1/2}\Big\|_p d\alpha
\end{align*}
which by 
 \eqref{dualLP} is bounded by
 \begin{align*}
  C(p-1)^{-2}&\biggl[
\Big\| \Big(\sum_{n\in\Z}\sum_{i=1}^{\cN_{n,\ell}(U)}\sum_{j\in\Z} |
 T_{j,\ell,+}^{2^{nb} s_{n,\ell}(i) } f|^2\Big)^{1/2}\Big\|_p
\\ &\,+ \int_{0}^{2^{-\ell} }
\Big\|
\Big(\sum_{n\in\Z}\sum_{i=1}^{\cN_{n,\ell}(U)} \sum_{j\in\Z}|
\frac{d}{d\alpha} T_{j,\ell,+}^{2^{nb} (s_{n,\ell}(i) +\alpha) } f|^2\Big)^{1/2}\Big\|_p d\alpha\biggr].
\end{align*}
From here on the estimation is exactly analogous to the 
previous square function -- just replace 
$A_{j,\ell}^u$ with $T_{j,\ell,+}^u$. 
The arguments for the corresponding terms 
with $T_{j,\ell,-}^u$ are similar (or could be reduced to the previous case by a change of variable, and curve). This concludes the proof of Theorem \ref{mainthm}.

\section{Lower bounds for $p\le 2$}\label{Mink-lower-bounds} 
As mentioned before the lower bound $(\log \fN(U))^{1/2}$ for  $\|\cH^U\|_{L^p\to L^p}$ , based on ideas of Karagulyan \cite{Kar07}, was established in \cite{grsy-first}.
We now show the easier lower bound 
in terms of the quantity 
$\sup_{\delta>0} \cK_p(U,\delta)$ 
(where we only have to consider the cases $\delta<1$).
The same calculation gives the same type of lower bound for $\|\cM^U\|_{L^p\to L^p}$.

By rescaling in the second variable and reflection we may assume that $c_+=1$. For $u\in U$ and $\delta\in (0,1)$ we define 
\[V_{\delta}(u)=
\{(x_1, x_2): 1\le x_1\le 2, |x_2-u x_1^b|\le 
\delta/4 \}\]
and let $f_\delta$ be the characteristic function of the ball of radius $\delta$ centered at the origin.
Observe that for $ 1\le x_1,u\le 2$ , $\eps<1$ 
and $x_1\le t\le x_1+\eps\delta $ we have
$u(t^b-x_1^b)\le 2b \cdot3^{b-1} \eps\delta$. Thus for 
$\eps_b= (8b \cdot3^{b-1})^{-1}$ we get 
 $f_\delta (x_1-t, x_2- u t^b)=1$ and thus
\[H^{(u)} f_\delta (x) \ge \frac 13 \int_{x_1}^{x_1+ \eps_b\delta}   f_\delta (x_1-t, x_2- u t^b) dt
\ge \frac{\eps_b}{3}  \delta,  \quad x\in V_{\delta}(u).\]
By rescaling in the second variable we have for every $r>0$ that $$\|\cH^{U}\|_{L^p\to L^p} \ge
\|\cH^{U^r}\|_{L^p\to L^p},$$
where $U^r=r^{-1}U\cap [1,2]$. Let $U^r(\delta)$ be a maximal $2^b\delta$--separated  subset of $U^r$, then $\# U^r(\delta) \gc N(U^r,\delta)$. This implies  
\[\cH^{U^r(\delta)} f_\delta (x) \gc\delta 
\text{ for }  x\in V_{r, \delta}:=
\bigcup_{u\in U^r(\delta)}V_{\delta}(u).\]
For different $u_1, u_2\in U^r(\delta)$ the sets 
$V_{\delta}(u_1)$ and $V_{\delta}(u_2)$
are disjoint and therefore we have $\meas(V_{r,\delta})\gc \delta\#(U_r(\delta))$. Hence  we get
$$\|\cH^{U^r(\delta)} f_\delta \|_p \ge c\delta^{1+1/p} \#(U_r(\delta))^{1/p}.$$ Since also 
$\|f_\delta\|_p\lc \delta^{2/p}$ we obtain 
$$\|\cH^{U}\|_{L^p\to L^p} \ge
\|\cH^{U^r(\delta)}\|_{L^p\to L^p} \gc \delta^{1-\frac 1p} \#(U^r(\delta))^{\frac 1p} 
\gc \delta^{1-\frac 1p} N(U^r,\delta)^{\frac 1p}$$ 
which gives the  uniform     lower bound
\Be\|\cH^U\|_{L^p\to L^p} \gc \cK_p(U,\delta)\Ee for sufficiently small $\delta$.  

\newcommand{\etalchar}[1]{$^{#1}$}

%\newpage

\begin{thebibliography}{22}

\bibitem{Carbery88}  Anthony Carbery. \emph{Differentiation in lacunary directions and an extension of the Marcinkiewicz multiplier theorem.} Ann. Inst. Fourier (Grenoble) 38 (1988), no. 1, 157--168. 

\bibitem%[CWW85]%
{CWW85} S.-Y. A. Chang; J.M. Wilson;  T.H. Wolff. \emph{Some weighted norm inequalities concerning the Schr\"odinger operators. } Comment. Math. Helv. 60 (1985), no. 2, 217-246.


\bibitem
{DD14} Ciprian Demeter;  Francesco Di Plinio. \emph{Logarithmic $L^p$ bounds for maximal directional singular integrals in the plane.} J. Geom. Anal. 24 (2014), no. 1, 375-416.


\bibitem{DGTZ18} Francesco Di Plinio; Shaoming  Guo; Christoph Thiele; Pavel  Zorin-Kranich.
	\emph{Square functions for bi-Lipschitz maps and directional operators.}
	J. Funct. Anal. 275 (2018), no. 8, 2015--2058.

\bibitem{GHLR17}  Shaoming Guo; Jonathan Hickman; Victor Lie; Joris Roos. \emph{Maximal operators and Hilbert transforms along variable non-flat homogeneous curves.} Proc. Lond. Math. Soc. (3) 115 (2017), no. 1, 177-219.

\bibitem{grsy-first} Shaoming Guo; Joris Roos; Andreas Seeger; Po-Lam Yung. \emph{A maximal function for families of Hilbert transforms along homogeneous curves.} Preprint, arXiv:1902.00096. Published online in  Math. Ann., doi.org/10.1007/00208-019-01915-3


 
\bibitem{Kar07} G.A. Karagulyan,  \emph{On unboundedness of maximal operators for directional Hilbert transforms.} Proc. Amer. Math. Soc. 135 (2007), no. 10, 3133-3141.

\bibitem{MR} Gianfranco Marletta; Fulvio Ricci. \emph{Two-parameter maximal functions associated with homogeneous surfaces in $\bbR^n$.} Studia Math. 130 (1998), no. 1, 53-65.

\bibitem{NSW}  Alexander Nagel;   Elias M. Stein; Stephen Wainger. \emph{ Differentiation in lacunary directions.} Proc. Nat. Acad. Sci. U.S.A. 75 (1978), no. 3, 1060--1062. 

\bibitem{STW03} 
Andreas Seeger; Terence  Tao; James  Wright. \emph{Endpoint mapping properties of spherical maximal operators.} J. Inst. Math. Jussieu 2 (2003), no. 1, 109-144. 

\bibitem{SWW95} 
 Andreas Seeger; Stephen Wainger;  James Wright. \emph{Pointwise convergence of spherical means.} Math. Proc. Cambridge Philos. Soc. 118 (1995), no. 1, 115-124. 

\bibitem{SWW97} \bysame.
%Seeger, Andreas; Wainger, Stephen; Wright, James. 
\emph{Spherical maximal operators on radial functions.} Math. Nachr. 187 (1997), 241--265. 

%\bibitem{Ste93} Elias M. Stein. Harmonic analysis: real-variable methods, orthogonality, and oscillatory integrals. With the assistance of Timothy S. Murphy. Princeton Mathematical Series, 43. Monographs in Harmonic Analysis, III. Princeton University Press, Princeton, NJ, 1993.

\end{thebibliography}
\end{document}